\documentclass{article}
 
\usepackage{amsmath, enumerate, a4wide}
\usepackage{amsfonts, amsthm}
\usepackage{amssymb}
\usepackage{verbatim}
\usepackage[colorlinks]{hyperref}

\newtheorem{lemma}{Lemma}
\newtheorem{claim}{Claim}
\newtheorem{problem}{Problem}
\newtheorem{conjecture}{Conjecture}
\newtheorem{theorem}[lemma]{Theorem}
\newtheorem{definition}[lemma]{Definition}

\newcommand{\kex}{{k\text{-\sc{ex}}}}
\newcommand{\ex}{{0\text{-\sc{ex}}}}
\newcommand{\nex}{{\text{\sc{ex}}}} 

\newcommand{\glex}{\succ_{\text{\sc{lex}}}}

\title{Rota's Basis Conjecture holds asymptotically}
\author{Alexey Pokrovskiy\thanks{Department of Mathematics, University College London, {\tt dr.alexey.pokrovskiy@gmail.com}.}}
\date{}

\begin{document}

\maketitle

\begin{abstract}
Rota's Basis Conjecture is a well known problem from matroid theory, that states that for any collection of $n$ bases in a rank $n$ matroid, it is possible to decompose all the elements into $n$ disjoint rainbow bases. 
Here an asymptotic version of this is proved. We show that it is possible to find $n-o(n)$ disjoint rainbow independent sets of size $n-o(n)$.
\end{abstract}
\section{Introduction}
In 1989, Rota made the following conjecture ``in any family $B_1, \dots, B_n$ of $n$ bases in  a vector space $V$, it is possible to find $n$ disjoint \emph{rainbow bases}'' (see \cite{huang1994relations}, Conjecture 4). Here a \emph{rainbow basis} means a basis of $V$ consisting of precisely one vector from each of $B_1, \dots, B_n$. In the context of this conjecture ``disjoint'' means that we do not have two rainbow bases using the same vector from the same basis $B_i$. Rota's conjecture has attracted attention due to its simplicity and connections to apparently unrelated areas. For example Huang and Rota \cite{huang1994relations} found connections between it and problems about Latin squares and supersymetric bracket algebra. Amongst other things, the recent collaborative Polymath project \cite{PolymathProposal} studied an approach to Rota's conjecture using topological tools.

It was observed (by Rota as well), that this might hold in the much more general setting of matroids rather than vector spaces. Matroids are an abstraction of independent sets in vector spaces, which also generalize many other ``independence structures''. They are defined on a set $V$ called the \emph{ground set} of the matroid. A matroid $M$ is a nonempty family of subsets of $V$ (called independent sets) which is closed under taking subsets and satisfies the following additional property (called the ``augmentation property''): that if $I, I'\in M$ are two independent sets with $|I|>|I'|$, then there is some element $x\in I\setminus I'$ such that $I'\cup \{x\}$ is also an independent set in $M$. A \emph{basis} of $M$ is a maximal independent sets. By the augmentation property all bases of $M$ must have the same size, which is called the \emph{rank} of $M$.
Using this terminology, the general Rota's Basis Conjecture (see \cite{huang1994relations}) can be phrased as:
\begin{conjecture}[Rota's Basis Conjecture]
Let $B_1, \dots, B_n$ be disjoint bases in a rank $n$ matroid $M$. Then it is possible to decompose $B_1\cup \dots\cup B_n$ into $n$ disjoint rainbow bases.
\end{conjecture}

Rota's conjecture attracted a lot of attention due to its simple formulation, and due to a large range of possible approaches towards it (coming from the many different settings in which matroids can naturally arise).  
One research direction is to prove the conjecture for some particular naturally-arising class of matroids. 
For example for matroids arising from real vector spaces (called real-representable matroids), the conjecture is known to hold whenever $n-1$ or $n+1$ is prime. This was proved in a combination of papers. First, Huang and Rota~\cite{huang1994relations} reduced the conjecture for real-representable matroids to the Alon-Tarsi Conjecture (which is a conjecture unrelated to matroids and states that for all $n$ the number of even and odd Latin squares of order $n$ is different). The Alon-Tarsi Conjecture was proved for $n-1$ prime by Drisko \cite{drisko1997number} and for $n+1$ prime by Glynn  \cite{glynn2010conjectures}.
Rota's Conjecture is known to hold for some other classes of matroids too. It was proved for paving matroids by Geelen and Humphries~\cite{geelen2006rota},  for strongly base orderable matroids by Wild~\cite{wild1994rota}, and for rank $\leq 4$ matroids computationally by Cheung~\cite{cheung2012computational}.

Another research direction   is to try and establish a weaker conclusion which holds for \emph{all matroids}. There are a number of natural approaches
  here:
\begin{enumerate}[(1)]
\item \textbf{Find many disjoint rainbow bases in $B_1\cup \dots\cup B_n$:} Finding one rainbow basis is easy using the augmentation property. Finding more is already challenging. Geelen and Webb~\cite{geelen2007rota} found $\Omega(\sqrt n)$ disjoint rainbow bases. This was improved to $\Omega(n/\log n)$ by Dong and Geelen~\cite{dong2019improved}, and further to $n/2-o(n)$ by Bucic, Kwan, Sudakov, and the author~\cite{bucic2018halfway}.
\item \textbf{Decompose $B_1\cup \dots\cup B_n$ into few rainbow independent sets:} The conjecture asks for a decomposition into $n$ independent sets.  Aharoni and Berger  showed that you can decompose into $2n$ independent sets. This was investigated further during the Polymath 12 project~\cite{PolymathPaper} where it was improved to $2n-2$.  
\item \textbf{Find $n$ disjoint rainbow independent sets of large total volume:} This means rainbow independent sets $I_1, \dots, I_n$ with $\sum_{i=1}^n |I_n|$ as large as possible. Rota's conjecture says that we can get $\sum_{i=1}^n |I_n|=n^2$. Both of the previous approaches give something here --- having $s$ rainbow bases clearly gives a family of independent sets of volume $sn$, whereas in a decomposition of $B_1\cup \dots\cup B_n$ into $t$ rainbow independent sets, the $n$ largest of these must have total volume at least $(n/t)n^2$. Thus the best known results about (1) and (2) both give a family of independent sets of volume around $n^2/2$. 
\end{enumerate}

In each of the above three approaches it is desirable to obtain an asymptotic version of the conjecture. In other words: Can you find $(1-o(1))n$ disjoint rainbow bases? Can you decompose $B_1\cup \dots\cup B_n$ into $(1+o(1))n$ disjoint rainbow independent sets? Can you find $n$ disjoint rainbow independent sets of  total volume $(1-o(1))n^2$? Previously such results were proved only for special classes of matroids --- Friedman and McGuinness~\cite{friedman2019girth} proved an asymptotic version for large girth matroids.  Combining the results of~\cite{huang1994relations, drisko1997number, glynn2010conjectures} with the Prime Number Theorem gives an asymptotic version for real-representable matroids~\cite{MattPrivate}. In this paper, we prove the first asymptotic version of the conjecture which holds \emph{for all matroids}.
\begin{theorem}\label{Theorem_main}
Let $B_1, \dots, B_n$ be disjoint bases in a rank $n$ matroid $M$. Then there are $n-o(n)$ disjoint rainbow independent sets in $B_1\cup \dots\cup B_n$  of size $n-o(n)$.
\end{theorem}
Notice that the union of these independent sets has size $(1-o(1))n^2$, and so this theorem gives an asymptotic version of the conjecture, when one takes approach (3) above. Going forward, it would be interesting to obtain stronger asymptotic versions of the conjecture as well as a proof for large rank matroids. Theorem~\ref{Theorem_main} is likely to be a good starting point in proving such results --- in recent years ``absorption techniques'' have been used in related problems to turn asymptotic solutions like Theorem~\ref{Theorem_main} into exact ones.
 
\section{Proof outline}
Here, we explain the ideas of our proof  by presenting a simplified version of it with some complications missing. Aside from some definitions, everything here is not used in the actual proof.

In this paper, we use the term \emph{``coloured matroids''} to mean a matroid with a colour assigned to each element in the ground set such that the colour classes are independent. 
 We will work with families  $\mathcal T=\{T_1, \dots, T_m\}$   of disjoint rainbow independent sets in a coloured matroid $M$. We use $E(\mathcal T)$ for the subset $T_1\cup \dots \cup T_m$ of the ground set of $M$. For two families  $\mathcal T=\{T_1, \dots, T_m\}$, $\mathcal S=\{S_1, \dots, S_m\}$, we say that $\mathcal T$ is a subfamily of $\mathcal S$, denoted  $\mathcal T\subseteq \mathcal S$, if $T_i\subseteq  S_i$ for all $i$.
 For a colour, we use $E_{\mathcal T}(c)$ to denote the set of colour $c$ elements of $\mathcal T$ and fix $e_{\mathcal T}(c)=|E_{\mathcal T}(c)|$. For a coloured matroid  $M$, use $C(M)$ for the set of colours occurring on $M$.

We call  a family $\mathcal T$  \emph{maximum} if $|E(\mathcal T)|$ is maximum amongst families of disjoint rainbow independent sets in $M$.
Rota's Conjecture is equivalent to saying that a maximum family has  $n^2$ elements. 
Theorem~\ref{Theorem_main} is equivalent to proving that a maximum family has $\geq (1-\epsilon)n^2$ elements. We achieve this by studying how elements can be moved between the rainbow independent sets of the family. The key definition is that of a ``reduced family of $\mathcal T$'' which informally means deleting all elements from $\mathcal T$ which can be moved around robustly.  
\begin{definition}[Reduced family]
Let $\mathcal{T}$ be a family of disjoint rainbow independent sets in a coloured matroid.
  Define the $\ell$-\emph{reduced family} $\mathcal T_{\ell}'$ of $\mathcal T$ to be $\mathcal T$ minus all elements $e\in E(\mathcal T)$ for which there are at least $\ell$ different choices of $T_j\in \mathcal T$ with $T_j+e$ a rainbow independent set. Define $\mathcal T_{\ell}^{(r)}={{\mathcal T_{\ell}''''}^{\dots}}$, where we repeat the operation $r$ times. We fix $\mathcal T_{\ell}^{(0)}=\mathcal T$.
\end{definition}
A maximum family of rainbow independent sets $\mathcal T$ has the property that no element $e$ outside $\mathcal T$ can be added to any rainbow independent set $T\in \mathcal T$ without breaking either rainbowness or independence. Our proof rests on this property being preserved by reduction.
\begin{lemma}\label{Lemma_sketch_switching}
Fix $r=0$ or $1$.
Let $M$ be a coloured matroid and $\mathcal T$ a family of disjoint rainbow independent sets in $M$. Suppose we have $T\in \mathcal T^{(r)}_3$ and $e\not\in \mathcal T$ with $T+e$ rainbow and independent. Then $\mathcal T$ is not maximum.
\end{lemma}
\begin{proof}
Suppose that $T+e$ rainbow and independent for some $T\in \mathcal T'_3$.
Let $R\in \mathcal T$ with $T\subseteq R$. Since $T+e$ is rainbow and independent, there are two (or possibly less) elements $f_1, f_2\in R\setminus T$ with $R-f_1-f_2+e$ rainbow and independent (choose $f_1$ with $c(f_1)=c(e)$ and use the augmentation property to get $f_2$ with $R-f_2+e$ independent). Since $f_1, f_2\not\in E(\mathcal T'_3)$, there are $Q_1, Q_2\in \mathcal T$ with $Q_1+f_1$ and $Q_2+f_2$ rainbow and independent. In fact from the definition of $3$-reduced, there are at least $3$ choices for $Q_1, Q_2$. This allows us to choose them to be distinct from each other and $R$. Now in $\mathcal T$ replace $R$ by $R-f_1-f_2+e$, $Q_1$ by $Q_1+f_1$, and $Q_2$ by $Q_2+f_2$. This gives a larger family.  
\end{proof}

The above lemma generalizes to arbitrary $r$ in Lemma~\ref{Lemma_switching} (by increasing ``$3$'' to something bigger). Our proof consists of showing that for any family $\mathcal T$ with $e(\mathcal T)\leq (1-\epsilon)n^2$, its reduction $\mathcal T^{(r)}_{\ell}$ becomes small for some $r$. 
 We do this step by step by showing the inequality ``$e(\mathcal T'_{\ell})\leq e(\mathcal T)-\epsilon^2 n^2$'' for such families $T$. By iterating this inequality we get $e(\mathcal T^{(r)}_{\ell})\leq e(\mathcal T)-r\epsilon^2 n^2$. 
 The proof of this inequality rests on the following lemma which estimates how many edges every colour loses when reducing the family. 
\begin{lemma}\label{Lemma_sketch_increment}
Fix $r=0$ or $1$.
Let $M$ be a coloured matroid  with $n$ colours of size $n$ and let $\mathcal T$ be a maximum family of $m$ disjoint rainbow independent sets in $M$. Let  $T\in \mathcal T^{(r)}_3$ and let $c$ be a colour missing from $T$. Then
$$e_{\mathcal T^{(r+1)}_1}(c)\leq  |T|-n+m$$
\end{lemma}
\begin{proof}
By the augmentation property there are $n-|T|$ colour $c$ elements $e$ with $T+e$ a rainbow independent set.  Lemma~\ref{Lemma_sketch_switching} tells us that these all occur on $\mathcal T$.
 By definition of $\mathcal T'_1$, they are all absent from $\mathcal T^{(r+1)}_1$. Additionally there are at least $n-m$ colour $c$ elements absent from $\mathcal T$, which remain absent in $\mathcal T^{(r+1)}_1$.
\end{proof}
In the actual proof use a version of this $\mathcal  T_{\ell}^{(r)}$ with larger $r, \ell$ (Lemma~\ref{Lemma_increment}).
The above lemma gives its best bound when $T$ is as small as possible. This motivates us to define the  \emph{excess of a colour $c$ in $\mathcal T$}:
$$\nex(c, \mathcal T):= \max(0,e_{\mathcal T}(c)+n-m-\min_{T\in \mathcal T: c\not\in T} |T|)$$ 
Lemma~\ref{Lemma_sketch_increment} now can be rephrased as saying that $e_{\mathcal T^{(r+1)}_1}(c)\leq e_{\mathcal T^{(r)}_3}(c)-\nex(c, \mathcal T^{(r)}_3)$ for every colour missing from some $T\in \mathcal T$. To get an inequality like ``$e(\mathcal T^{(r+1)})\leq e(\mathcal T^{(r)})-\epsilon^2 n^2$'' from this, we need to show that the average excess over all colours is $\epsilon^2 n$. We show the following:
\begin{lemma}\label{Lemma_sketch_excess}
Let $M$ be a coloured matroid  with $n$ colours of size $n$ and let $\mathcal T$ be a family of $(1-\epsilon)n$ disjoint rainbow independent sets in $M$.  
Then 
$$\frac{1}n\sum_{\text{colours } c}\nex(c, \mathcal T) \geq \epsilon^2 n$$
\end{lemma}
The proof of this has nothing to do with matroids or colours. The essence of it turns out to be a very short lemma about bipartite graphs (see Lemma~\ref{Lemma_0-excess}).
We now have all the ingredients that go in the Theorem~\ref{Theorem_main}. To summarize, the structure is:
\begin{itemize}
\item Start with  $M$, a coloured matroid  with $n$ colours of size $n$.
\item Consider a maximum family $\mathcal T$ of $(1-\epsilon)n$ disjoint rainbow independent sets in $M$.
\item Suppose for contradiction that $e(\mathcal T)\leq (1-\epsilon )n^2$.
\item By a variant of Lemma~\ref{Lemma_sketch_excess}, we have $\sum_{\text{colours } c}\nex(c, \mathcal T^{(r)}_{\ell}) \geq \epsilon^2 n^2$ for all $r$ and large enough $\ell$.
\item By a variant of Lemma~\ref{Lemma_sketch_increment}, and maximality, we have $e(\mathcal T^{(r)}_{\ell})< (1-r\epsilon^2)n^2$ for  all $r$ and large enough $\ell$. At $r=1/\epsilon^2$ this is a contradiction (meaning that the assumption  ``$\mathcal T$ is maximum'' in one of the applications of Lemma~\ref{Lemma_sketch_increment} along the way was invalid).
\end{itemize}
In this sketch, there are a couple of things missing. Most of them are easy to fill in --- namely  Lemmas~\ref{Lemma_sketch_switching} and~\ref{Lemma_sketch_increment} can be proved for larger $r$ and $\ell$. 

However there is one complication which appears to require significant changes to the above strategy --- namely the requirement that ``$c$ is missing from some $T\in \mathcal T$''   for the inequality  ``$e_{\mathcal T^{(r+1)}_1}(c)\leq e_{\mathcal T^{(r)}_3}(c)-\nex(c, \mathcal T^{(r)}_3)$'. When there are many colours that occur on all $T\in \mathcal T$, then it is possible that $\mathcal T'_1= \mathcal T$, which breaks the above strategy (as an example, consider a family $\mathcal T$ consisting of $(1-\epsilon)n$ rainbow independent sets of size $n/2$ all using the same $n/2$ colours and nothing else).  
The way we get around this issue is to change what ``maximum family'' means. Rather than asking them to have as many elements as possible, we instead ask them to be ``lexicographically maximum'' which means roughly that $\min_{\text{colours } c}e_{\mathcal T}(c)$ is as large as possible. The overall structure of the proof remains unchanged --- it follows analogues of the above lemmas with suitable changes.

\section{Proof of Theorem~\ref{Theorem_main}}
Rather than working with the excess of a family as in the proof outline, we will associate an auxiliary bipartite graph  to every family and study a parameter $\kex_G(y)$ associated to the graph.
\begin{definition}
Let $y$ be a vertex in a graph $G$. Let $N(y)=\{x_1, \dots, x_{d(y)}\}$ be ordered with $d(x_1)\geq d(x_2)\geq \dots\geq d(x_{d(y)})$. Define the $k$-excess of $y$ in $G$
$$ \kex_G(y)=\max(0, d(x_k)-d(y)).$$
\end{definition}
If a vertex $y$ has less than $k$ neighbours, then this definition says $\kex_G(y)=0$.
The following  lemma will imply Lemma~\ref{Lemma_sketch_excess}.
\begin{lemma}[$0$-excess sum]\label{Lemma_0-excess}
Let $G$ be a bipartite graph with parts $X,Y$. 
Let $\delta(Y)$ denote  the smallest degree in $G$ out of vertices of $Y$. Then 
$$\sum_{y\in Y}\ex_G(y)\geq \delta(Y)(|Y|-|X|)$$
\end{lemma}
\begin{proof} Let $M$ be a maximum matching in $G$ and $C$ a minimum vertex cover. By K\"onig's Theorem we have $e(M)=|C|$ and so each edge of $M$ contains precisely one vertex of $C$. 
In particular $M\cap Y\setminus C$ is matched to $M\cap X\cap C$.
Since $C$ is a vertex cover, we have $N(Y\setminus C)\subseteq C\cap M\cap X$. This gives
\begin{align*}
\sum_{y\in M\cap Y\setminus C}\ex_G(y)
&\geq \sum_{x\in M\cap X\cap C}d(x)- \sum_{y\in M\cap Y\setminus C}d(y)
\geq e(M\cap X\cap C,Y\setminus C)-e(M\cap Y\setminus C, X)\\
&= e(Y\setminus M, X)
\geq \delta(Y)|Y\setminus M|\geq \delta(Y)(|Y|-|X|).
\end{align*}
\end{proof}

\begin{lemma}[$k$-excess sum]\label{Lemma_bipartite}
Let $G$ be a bipartite graph with parts $X,Y$. For a subset $Y'\subseteq Y$, let $\delta(Y')$ be the smallest degree in $G$ out of vertices of $Y'$. 
Then 
$$\sum_{y\in Y'}\kex_G(y)\geq \delta(Y')(|Y'|-|X|)-2k|Y'|$$
\end{lemma}  
\begin{proof}
 Obtain a subgraph $G'$ as follows: delete all vertices of $Y\setminus Y'$. For each $y\in Y'$ delete $k$ edges going to the $k$ vertices of largest degrees in $N(y)$. Notice that $\kex_G(y)\geq \ex_{G'}(y)-k$ for each $y\in Y'$ and also $\delta_{G'}(Y')= \delta_G(Y')-k$. The result follows from Lemma~\ref{Lemma_0-excess} applied to $G'$.
\end{proof}

We associate a bipartite graph to every family $\mathcal T$.
\begin{definition}[Availability graph]
Let $\mathcal T=\{T_1, \dots, T_{m}\}$ be a family of rainbow independent sets in a coloured matroid $M$. The \emph{availability graph}  of $\mathcal T$, denoted $A(\mathcal T)$, is the  bipartite graph with parts $\{T_1, \dots, T_m\}$ and $C(M)$, and with $T_ic_j$ an edge of $A(\mathcal T)$ whenever $c_j\not\in C(T_i)$. 
\end{definition}

Notice that the degree $d_{A(\mathcal T)}(T_i)$ is the number of colours missing from $T_i$ and the degree $d_{A(\mathcal T)}(c)$ is the number of independent sets missing $c$. The two different definitions of excess that we introduced should now make sense because we have $\nex(c,\mathcal T)=\ex_{A(\mathcal T)}(c)$ for any colour missing from some $T\in \mathcal T$ (whereas for colours present on all $T\in \mathcal T$, the definitions disagree since we have $\nex(c,\mathcal T)=n$ and $\ex_{A(\mathcal T)}(c)=0$). Lemma~\ref{Lemma_sketch_excess}  can now easily be deduced from Lemma~\ref{Lemma_0-excess} (although it is not used in the proof). The following is the  analogue of Lemma~\ref{Lemma_sketch_increment} we use.

\begin{lemma}[Increment lemma]\label{Lemma_increment} 
Let $\mathcal T$ be a family of $\leq n$ rainbow independent sets in a coloured matroid $M$ with $n$ colours of size $\geq n$,  and $c$ a colour. At least one of the following holds:
\begin{enumerate}[(i)]
\item There is some $T_i\in \mathcal T$ for which there are at least $d_{A(\mathcal T)}(c)+\frac12\kex_{A(\mathcal{T})}(c)$ colour $c$ elements $e\not\in E(\mathcal T)$ with $T_i+e$ a rainbow independent set.
\item The $\ell$-reduced family has $e_{\mathcal{T_{\ell}'}}(c)\leq e_{\mathcal T}(c)-\frac12\kex_{A(\mathcal T)}(c)+ \ell n/k$.
\end{enumerate}
\end{lemma}
\begin{proof}
If  $\kex_{A(\mathcal T)}(c)=0$, then (ii) is trivially true, so we can assume that  $\kex_{A(\mathcal T)}(c)>0$.
By the definition of $\kex_{A(\mathcal T)}(c)$, there are $k$ rainbow independent sets $T_1,\dots, T_k\in \mathcal T$ with each $T_i$ missing colour $c$ and each $T_i$ missing at least $d_{A(\mathcal T)}(c)+\kex_{A(\mathcal T)}(c)$ colours in total. 
Define a bipartite graph $H$ whose parts are $X=\{T_1,\dots, T_k\}$ and $E(c)$ with $T_ie$ an edge whenever $T_i+e$ is a (rainbow) independent set. Using the augmentation property and the fact that $M$ consists of $n$ colours of size $\geq n$, we have $d_H(T_i)\geq |E(c)|-|T_i|\geq n-|T_i|= d_{A(\mathcal T)}(T_i)\geq d_{A(\mathcal T)}(c)+\kex_{A(\mathcal T)}(c)$  for each $i=1, \dots, k$. Equivalently, there are at least $d_{A(\mathcal T)}(c)+\kex_{A(\mathcal T)}(c)$ colour $c$ elements $e$ with $T_i+e$ a rainbow independent set. If, for some $i=1, \dots, k$, at least $d_{A(\mathcal T)}(c)+\frac12\kex_{A(\mathcal T)}(c)$ of these have $e\not\in E(\mathcal T)$, then case (i) of the lemma holds.

Thus we can assume that for all $i=1, \dots, k$, there are $\geq \frac12\kex_{A(\mathcal T)}(c)$ colour $c$ elements $e\in E(\mathcal T)$ with $T_i+e$ a rainbow independent set. Let $H'$ be the induced subgraph of $H$ on $X=\{T_1,\dots, T_k\}$ and $E(c)\cap E(\mathcal T)$ (so we have $\delta_{H'}(X)\geq \frac12\kex_{A(\mathcal T)}(c)$). 
Let  $E_{\geq \ell}\subseteq{ E(c)\cap E(\mathcal T)}$ be the set of elements $e$ with $d_{H'}(e)\geq \ell$.   We have 
$$|X||E_{\geq \ell}|+\ell|E(c)\cap E(\mathcal T)|\geq \sum_{e\in E_{\geq \ell}}|X|+\sum_{e\in E(c)\cap E(\mathcal T)\setminus E_{\geq \ell}}\ell\geq e(H')\geq |X|\frac12\kex_{A(\mathcal T)}(c).$$
Using $|X|=k$ and rearranging gives $|E_{\geq \ell}|\geq \frac12\kex_{A(\mathcal T)}(c)- \ell|E(c)\cap E(\mathcal T)|/k\geq \frac12\kex_{A(\mathcal T)}(c)-\ell n/k$.
From the definition of the $\ell$-reduced family $\mathcal T'_{\ell}$,   we have $E_{\mathcal T'_{\ell}}(c)=E_{\mathcal T}(c)\setminus E_{\geq \ell}$, implying (ii).
  \end{proof}

The following is the  analogue of Lemma~\ref{Lemma_sketch_switching} we use. For technical reasons there are two families $\mathcal S,\mathcal T$ in this lemma, but   the most important case is when $\mathcal S=\mathcal T$. In that case the lemma is  an extension of Lemma~\ref{Lemma_sketch_switching} to larger $r,\ell$.
\begin{lemma}[Switching lemma]\label{Lemma_switching}
Let $\ell \geq 9^{r}$. 
Let $\mathcal T=\{T_1, \dots, T_m\}$, $\mathcal S=\{S_1, \dots, S_m\}$ be two families of disjoint rainbow independent sets in a coloured matroid $M$ with $T_i\subseteq S_i$ for all $i$. 
Suppose we have  an $e \not\in E(\mathcal S)$ with $T_0+e$ rainbow and independent for some $T_0\in \mathcal T_{\ell}^{(r)}$. 
Then there is a family of disjoint rainbow independent sets $\mathcal S^*$ with  $E(\mathcal S^*)=\{e\}\cup( E(\mathcal S)\setminus X)$, for some   $X\subseteq  E(\mathcal S)\setminus E(\mathcal T)$  with $|X|\leq 2^{r+1}$.
\end{lemma}  
\begin{proof}
Let $S_0\in \mathcal S$, $T_0'\in \mathcal T$  be the sets with $T_0\subseteq T_0'\subseteq S_0$.
We prove that additionally one can ensure the following:
\begin{enumerate}[(i)]
\item $\mathcal S$ and $\mathcal S^*$ differ on at most $3^r$ independent sets.
\item For any subfamily $\mathcal R\subset \mathcal S$ of $\ell/2^{r}$ independent sets with $S_0\not\in\mathcal R$, we can ensure that $\mathcal R$ is a subfamily of $\mathcal S^*$ also.  
\end{enumerate}
The proof is by induction on $r$ (with $\ell$ being fixed for the duration of the proof. So we prove the cases $r=0,1,\dots, \lfloor \log_9 \ell\rfloor$ in order). For the initial case ``$r=0$'', notice that using the augmentation property there are $e_1, e_2\in S_0$ with $S_0\setminus\{e_1, e_2\}\cup\{e\}$ rainbow and independent. Not the family $S^*=S-S_0+S_0\setminus\{e_1, e_2\}\cup\{e\}$ satisfies the lemma (recalling that ``$\mathcal T^{(0)}_{\ell}$'' just means $\mathcal T$). 

Now suppose that the lemma holds for some $r$.
Let $e\not\in \mathcal S$ with $e+T_0$ rainbow and independent for some $T_0\in E(\mathcal T^{(r+1)}_{\ell})$ and let  $\mathcal R\subset \mathcal S$ be a subfamily of   $\leq \ell/2^{r+1}$ independent sets with $S_0\not\in \mathcal R$. 
Using the augmentation property, there are two (or one) elements $e_1, e_2\in E(\mathcal S) \setminus E(\mathcal T_{\ell}^{(r+1)})$ with 
$S_0-e_1-e_2+e$ rainbow and independent.

We claim that there is a family $\mathcal S^*_1$ with $S_0\setminus e_1\in S^*_1$,
 having $E(\mathcal S^*_1)=E(\mathcal S)\setminus X_1$ for some  $X_1\subseteq  E(\mathcal S)\setminus E(\mathcal T)$  with $|X_1|\leq 2^{r+1}$, and $\mathcal S^*_1$ agreeing with $\mathcal S$ on all sets of $\mathcal R$.
 If $e_1\not\in E(\mathcal T)$, then we can use $\mathcal S^*_1=\mathcal S- S_0+S_0\setminus \{e_1\}$, so suppose $e_1\in E(\mathcal T)$.
Fix $\hat{\mathcal S}=\mathcal S - S_0+S_0\setminus \{e_1\}$ and $\hat{\mathcal T}=\mathcal T - T_0'+T_0'\setminus \{e_1\}$. 
Choose $r_1\leq r$ with $e_1\not\in E({\mathcal T}^{(r_1+1)}_{\ell})$ and $e_1\in E({\mathcal T}^{(r_1)}_{\ell})$. By definition of the $\ell$-reduced family $\mathcal T^{(r_1+1)}_{\ell}$, there are $\ell$ choices of $T_j\in \mathcal \mathcal T^{(r_1)}_{\ell}$ with $T_j+e_1$ a rainbow independent set. Since $\ell>|\mathcal R|$, we can choose such a $T_j$ which is outside $\mathcal R\cup\{S_0\}$. Let $\hat T_j$ be the corresponding independent set of $\hat{\mathcal T}_{\ell}^{(r_1)}$ and notice that $\hat T_j\subseteq T_j$ (it is easy to see that reduction is monotone in the sense that for two families with $\hat {\mathcal T}\subseteq \mathcal T$, we always have  $\hat {\mathcal T}^{(r)}_{\ell}\subseteq \mathcal T^{(r)}_{\ell}$.
Notice that $e_1\not\in E(\hat{\mathcal S})$ and has $e_1+\hat {T_j}$ rainbow and independent. By induction,  there is a family of disjoint rainbow independent sets $ {\mathcal S}_1^*$ with $e_1\in E({\mathcal S}_1^*)$,  $E(\hat {\mathcal S})\setminus E({\mathcal S}_1^*)\subseteq  E(\hat {\mathcal S})\setminus E(\hat {\mathcal T})=E({\mathcal S})\setminus E({\mathcal T})$  and $|E(\hat {\mathcal S})\setminus E({\mathcal S}_1^*)|\leq 2^{r_1+1}$. Additionally, $\hat {\mathcal S}$ and $ {\mathcal S}_1^*$ differ on at most $3^{r_1}$ independent sets, and agree on all independent sets of $\mathcal R_1\cup\{S_0\setminus e_1\}$. 

Similarly,  there is a family $\mathcal S^*_2$ with $S_0\setminus e_2\in S^*_2$,
 having $E(\mathcal S^*_2)=E(\mathcal S)\setminus X_2$ for some  $X_2\subseteq  E(\mathcal S)\setminus E(\mathcal T)$  with $|X_2|\leq 2^{r+1}$, and $\mathcal S^*_2$  agreeing with $\mathcal S$ on all sets of $\mathcal R\cup  (\mathcal S\setminus \mathcal S^*_1)\setminus\{S_0\}$ (using that there was room in the inequality $\ell>|\mathcal R|$ to increase $\mathcal R$ by the $\leq 3^r$ members of $(\mathcal S\setminus \mathcal S^*_1)\setminus\{S_0\}$). 
Note that on every set $S_i\in \mathcal S$, other than $S_0$, at most one of $\mathcal S^*_1/\mathcal S^*_2$ differs from $\mathcal S$. Construct  $\mathcal S^*$ from $\mathcal S$ by replacing $S_0$ by $S_0\setminus\{e_1,e_2\}\cup \{e\}$, and otherwise replacing every set of $\mathcal S^*$ by the corresponding set of $\mathcal S^*_1/\mathcal S^*_2$ when they differ.
 Formally $\mathcal S^*= \mathcal S^*_1\setminus (\mathcal S\cup\{S_0\setminus e_1\})+ \mathcal S^*_2\setminus (\mathcal S\cup\{S_0\setminus e_2\})+ \mathcal S^*_1\cap \mathcal S^*_2\cap \mathcal S+\{S_0\setminus\{e_1,e_2\}\cup \{e\}\}$. This family satisfies the lemma  and (i), (ii) ($e\in \mathcal S^*$ holds by construction, whereas the other properties come from the fact that the only sets on which $\mathcal S^*$ differs from $\mathcal S$ were ones where $\mathcal S^*_1/\mathcal S^*_2$ differed from $\mathcal S$).  
\end{proof}

We are now ready to show that Rota's conjecture holds asymptotically.
\begin{proof}[Proof of Theorem~\ref{Theorem_main}]
Fix $r_0=100/\epsilon^2$, $\ell =9^{r_0}$,  $k=\epsilon^{-2}\ell /16$, and $n\geq 100^{r_0}$.  Define a partial order $\glex$ on  families of disjoint rainbow independent sets $\mathcal S, \mathcal T$: Order the colours of $\mathcal S$ as $c_1, \dots, c_n$ with $e_{\mathcal S}(c_1)\leq \dots \leq e_{\mathcal S}(c_n)$, and order the colours of $\mathcal T$ as $d_1, \dots, d_n$ with $e_{\mathcal T}(d_1)\leq \dots \leq e_{\mathcal T}(d_n)$. We say $\mathcal S\glex \mathcal T$ if the smallest index $q$ with $e_{\mathcal S}(c_q)\neq e_{\mathcal T}(d_q)$ has $e_{\mathcal S}(c_q)> e_{\mathcal T}(d_q)$. Notice that ``$\glex$'' is a partial order.
 
 Let $M$ be a coloured matroid with $n$ colours of size $n$. Let $\mathcal R$ be a family of $(1-\epsilon)n$ disjoint rainbow independent sets in $M$ which is maximal with respect to $\glex$. Suppose for the sake of contradiction that there are at more than $\epsilon n$ colours $c$ with $e_{\mathcal R}(c)\leq (1-3\epsilon)n$. 

Order the colours $c_1, \dots, c_n$ with $e_{\mathcal R}(c_1)\leq e_{\mathcal R}(c_2)\leq \dots \leq e_{\mathcal R}(c_n)$. Choose $m$ to be the largest index with $e_{\mathcal R}(c_m)\leq (1-2\epsilon)n$ and $e_{\mathcal{R}}(c_{m+(3^{r_0})/\epsilon})\geq e_{\mathcal R}(c_m)+3^{r_0}$. To see that such an $m$ does indeed exist --- notice that either $m=n$ works, or $e_{\mathcal{R}}(c_n)> (1-2\epsilon)n$ and for each $i$ $e_{\mathcal{R}}(c_{i+(3^{r_0})/\epsilon})-3^{r_0}< e_{\mathcal R}(c_i)$. In the latter case we get
$e_{\mathcal{R}}(c_1)\geq e_{\mathcal{R}}(c_n)-3^{r_0}\frac{n}{3^{r_0}/\epsilon}> (1-3\epsilon)n$ which is a contradiction.  

We call the colours $c_1, \dots, c_{m-1}$ \emph{small}, the colours $c_{m}, \dots, c_{m+(3^{r_0})/\epsilon}$ \emph{medium}, and the colours $c_{m+(3^{r_0})/\epsilon+1}, \dots, c_n$ \emph{large}. By assumption, there are $> \epsilon n$ small colours. For a family $\mathcal F$, use $E_{\mathrm{small}}(\mathcal F)$/$E_{\mathrm{medium}}(\mathcal F)$/$E_{\mathrm{large}}(\mathcal F)$ to denote the sets of elements of corresponding colours in $\mathcal F$.  

\begin{claim} \label{Theorem_claim}
There is a family of disjoint rainbow independent sets ${\mathcal R}^*$ with $E_{\mathrm{small}}(\mathcal R^*) = E_{\mathrm{small}}(\mathcal R^*)+e$ for some element $e$ outside $\mathcal R$, $E_{\mathrm{medium}}(\mathcal R^*) = E_{\mathrm{medium}}(\mathcal R^*)$, and $|E_{\mathrm{large}}(\mathcal R)\setminus E_{\mathrm{large}}(\mathcal R^*)|\leq 2^{r_0+1}$.
\end{claim}
This claim implies the theorem since it implies  ${\mathcal R}^*\glex \mathcal R$ (contradicting maximality of ${\mathcal R}$). To see this notice that all large and medium colours have at least $e_{\mathcal R}(c_m)$ elements in $\mathcal R^*$ (for medium colours this happens because they are unchanged, whereas for large colours it happens because they lose at most $3^{r_0}$ elements, but initially had at least $e_{\mathcal{R}}(c_{m+(3^{r_0})/\epsilon})\geq e_{\mathcal R}(c_m)+3^{r_0}$ elements). However there is one extra small colour element $e$ in $\mathcal R^*$ (so the index $q$ in the definition of $R^*\glex R$ will be the one corresponding to $c(e)$). 
\begin{proof}[Proof of Claim~\ref{Theorem_claim}]
The basic idea of the proof is to apply the previous lemmas to $\mathcal R$ in order to find a small colour element which extends it. However we do not apply the lemmas to $\mathcal R$ directly (in order to avoid medium colours from being affected).
Instead we add ``dummy elements'' of  medium colours to independent sets from $\mathcal R$ in order to obtain a new family $\mathcal S$ with the property that every  $S\in \mathcal S$ contains every medium colour. These dummy elements need not come from $M$ --- we enlarge $M$ by adding as many new medium colour dummy elements as are needed in an arbitrary fashion (eg.   let the new dummy elements be independent from everything else).  

Let $\mathcal T$ be formed from $\mathcal S$ by deleting all large colour elements. Notice that $A(\mathcal T)$ has parts of size $|\mathcal T|=(1-\epsilon)n$ and $|C(M)|=n$ with all small colours having $d_{A(\mathcal T)}(c)\geq \epsilon n$, all medium colours having $d_{A(\mathcal T)}(c)=0$, and all large colours having $d_{A(\mathcal T)}(c)= (1-\epsilon)n\geq \epsilon n$. By Lemma~\ref{Lemma_bipartite} we have  
$$\sum_{c\in C(M)}\kex_{A(\mathcal T)}(c)\geq \sum_{c \text{ small/large}}\kex_{A(\mathcal T)}(c)\geq \epsilon n(|C(M)|-3^{r_0}/\epsilon -|\mathcal T|)- 2kn\geq \epsilon^2n^2/2$$ 
Additionally all steps in this calculation work  for any family $\mathcal T^*$ formed by deleting elements from $\mathcal T$ (since we'd have $A(\mathcal T)\subseteq A(\mathcal T^*)$). We claim that the following is true:

\begin{itemize}
\item[P:] There is some $r\leq r_0$, some small colour $c$, some colour $c$ element $e\not\in E_{\mathcal S}(c)$, and some independent set $T\in \mathcal T^{(r)}_{\ell}$ with $T+e$ a rainbow independent set.
\end{itemize}
There are two cases. First suppose that for some $r\leq r_0$ there is some $T\in \mathcal T^{(r)}_{\ell}$ with $d_{A(\mathcal T^{(r)}_{\ell})}(T)\geq (1-\epsilon)n$. Equivalently $|T|\leq \epsilon n$. Since there are $>\epsilon n$ small colours, there is some small colour $c$ absent from $T$. By definition of ``small colour'', $e_{\mathcal S}(c)\leq (1-2\epsilon)n$. By the augmentation property one of the $\geq 2\epsilon n$ colour $c$ elements outside $\mathcal S$ is independent from $T$. Let $e$ be such an element.

Now suppose that for all  $r\leq r_0$ we have  $d_{A(\mathcal T^{(r)}_{\ell})}(T)< (1-\epsilon)n$ for all $T$. Notice that in the $\ell$-reduced families $\mathcal T_{\ell}^{(0)}, \mathcal T^{(1)}_{\ell}, \mathcal T^{(2)}_{\ell}, \dots, \mathcal T^{(r_0)}_{\ell}$, no large colours occur (since these families are contained in $\mathcal T$), and all medium colours occur on all independent sets (from the definition of ``reduced family'', if a colour occurs on all $T\in \mathcal T$ then it also occurs on all $T\in \mathcal T'_{\ell}$). An immediate consequence of this is that every  medium and large colour $c$ has $k$-excess zero  in  all $A(\mathcal T^{(r)}_{\ell})$ (for large colours this comes from $d_{A(\mathcal T^{(r)}_{\ell})}(T)< (1-\epsilon)n$ for all $T$).

Apply Lemma~\ref{Lemma_increment} to the reduced families $\mathcal T^{(0)}_{\ell}, \mathcal T^{(1)}_{\ell}, \mathcal T^{(2)}_{\ell}, \dots, \mathcal T^{(r_0)}_{\ell}$ and to every  colour whose $k$-excess is positive in $A(\mathcal T^{(i)}_{\ell})$.
 We claim that for at least one of these applications case (i) of Lemma~\ref{Lemma_increment} has to occur. Indeed, otherwise we would have $e_{\mathcal{T}^{(i+1)}_{\ell}}(c)\leq e_{\mathcal T^{(i)}_{\ell}}(c)-\frac12\kex_{A(\mathcal T^{(i)}_{\ell})}(c)+ \ell n/k$ for all colours $c$ and all $i=0,1,\dots, r_0$ (for positive excess colours this will be from Lemma~\ref{Lemma_increment}. For zero excess colours it is trivial). Summing over all colours this would give $$e(\mathcal T^{(i+1)}_{\ell})\leq e(\mathcal T^{(i)}_{\ell})+\ell n^2/k-\sum_{c\in C(M)}\frac12\kex_{A(\mathcal T^{(i)}_{\ell})}(c)\leq e(\mathcal T^{(i)}_{\ell})+\ell n^2/k-\epsilon^2 n^2/4\leq e(\mathcal T^{(i)}_{\ell})-\epsilon^2 n^2/8$$ 
 This implies  $e(\mathcal T^{(r_0)}_{\ell})<0$ which is a contradiction.

Thus there is some $r\leq r_0$, some colour $c$ of positive $k$-excess in $A(\mathcal T^{(r)}_{\ell})$, and some independent set $T\in \mathcal T^{(r)}_{\ell}$ with at least $d_{A(\mathcal T^{(r)}_{\ell})}(c)+\frac12\kex_{A(\mathcal T^{(r)}_{\ell})}(c)>d_{A(\mathcal T^{(r)}_{\ell})}(c)$ colour $c$ elements $e\not\in E(\mathcal T^{(r)}_{\ell})$ with $T+e$ a rainbow independent set. Since medium/large colours have zero excess, $c$ is small. By the definition of the availability graph, we have $d_{A(\mathcal T^{(r)}_{\ell})}(c)=(1-\epsilon)n-|E_{\mathcal T^{(r)}_{\ell}}(c)|\geq |E_{\mathcal S\setminus \mathcal T^{(r)}_{\ell}}(c)|$ (since there are at most $(1-\epsilon)n$ colour $c$ elements of $\mathcal S$). Thus there is at least one colour $c$ element $e\not\in E(\mathcal S)$ with $T+e$ a rainbow independent set.

Now, having established that P is true, let $r,c,e,T$ be as in P.
Let $S\in \mathcal{S}$ with $T\subseteq S$. 
By Lemma~\ref{Lemma_switching}, there is a family of disjoint rainbow independent sets $\mathcal S^*$ with  $E(\mathcal S^*)=\{e\}\cup( E(\mathcal S)\setminus X)$, for some   $X\subseteq  E(\mathcal S)\setminus E(\mathcal T)$  with $|X|\leq 2^{r+1}$. Let $\mathcal R^*$ be $\mathcal S^{*}$ with all dummy elements deleted.
Since all elements in $E(\mathcal S)\setminus E(\mathcal T)$ were large, we have that small and medium colours are unchanged moving from $\mathcal R$ to $\mathcal R^*$ (other than colour $c$ gaining $e$). So $\mathcal R^*$  satisfies the claim.
\end{proof}
\end{proof}

\section{Concluding remarks}
It is easy to work out the bounds our proof gives: it produces $n-\frac{Cn}{\sqrt{\log n}}$ disjoint rainbow independent sets of size  $n-\frac{Cn}{\sqrt{\log n}}$ (for some fixed large constant $C$). It would be interesting to improve this. 
Additionally, it would be nice  to prove  qualitatively stronger asymptotic versions of the conjecture. 
The following problems are natural goals.
\begin{problem}
Let $B_1, \dots, B_n$ be disjoint bases in a rank $n$ matroid $M$.  Show that there are $(1-o(1))n$ disjoint rainbow bases.
\end{problem}
\begin{problem}
Let $B_1, \dots, B_n$ be disjoint bases in a rank $n$ matroid $M$.
Show that $B_1\cup \dots\cup B_n$ can be decomposed into $(1+o(1))n$ disjoint rainbow independent sets. 
\end{problem}
\begin{problem}
Let $B_1, \dots, B_n$ be disjoint bases in a rank $n$ matroid $M$.
Show that there are $n$ disjoint rainbow independent sets of size $(1-o(1))n$. 
\end{problem}
A solution to any of the above problems, would give a strengthening of Theorem~\ref{Theorem_main}. Theorem~\ref{Theorem_main} may be a good starting point for solving the above problems. It is not uncommon in combinatorics for non-trivial reductions between different kinds of asymptotic results to exist. Moreover, the results in this paper may eventually lead to a solution of Rota's Conjecture for sufficiently large $n$ via \emph{the absorption method}. Absorption is a technique for turning asymptotic results into exact ones. It has recently has found success in rainbow problems related to Rota's Conjecture~\cite{glock2020decompositions}. Now that we have an asymptotic solution to the conjecture in Theorem~\ref{Theorem_main}, it seems promising to try and turn it into a exact solution using absorption.

\subsubsection*{Acknowledgment}
The author would like to thank Matija Bucic, Matthew Kwan, and Benny Sudakov for discussions relating to this paper, particularly related to Lemma~\ref{Lemma_0-excess}. Additionally, he would like to thank Matthew Kwan for reading a draft of the proof and many suggested improvements.

\bibliographystyle{abbrv}
\bibliography{Rota}
\end{document}